\documentclass[12pt]{amsart}
\usepackage{amscd,amssymb,graphics}

\usepackage{amsfonts}
\usepackage{amsmath}
\usepackage{amsxtra}
\usepackage{latexsym}
\usepackage[mathcal]{eucal}
\usepackage[all]{xy}

\usepackage{graphics,colortbl}

\input xy
\xyoption{all}
\usepackage{epsfig}

\usepackage[pdftex,bookmarks,colorlinks,breaklinks]{hyperref}

\newtheorem{theorem}{Theorem}[section]
\newtheorem{fact}[theorem]{Fact}
\newtheorem{corollary}[theorem]{Corollary}
\newtheorem{lemma}[theorem]{Lemma}
\newtheorem{proposition}[theorem]{Proposition}

\newtheorem{question}[theorem]{Question}

\newtheorem{definition}[theorem]{Definition}
\numberwithin{equation}{section}

\theoremstyle{remark}
\newtheorem{remark}[theorem]{Remark}
\newtheorem{example}[theorem]{Example}

\newcommand{\ben}{\begin{enumerate}}
	\newcommand{\een}{\end{enumerate}}
\newcommand{\bit}{\begin{itemize}}
	\newcommand{\eit}{\end{itemize}}

\def\R {{\Bbb R}}
\def\Q {{\Bbb Q}}
\def \F {{\Bbb F}}

\def\C {{\Bbb C}}
\def\N{{\Bbb N}}
\def\T{{\Bbb T}}

\def\F{{\Bbb F}}

\def\Aut{{\mathrm Aut}\,}

\def\QED{\nobreak\quad\ifmmode\roman{Q.E.D.}\else{\rm Q.E.D.}\fi}

\def\hull#1{\langle#1\rangle}

\def\PSL{\operatorname{PSL}}
\def\GL{\operatorname{GL}}
\def\SL{\operatorname{SL}}
\def\UT{\operatorname{UT}}
\def\SO{\operatorname{SO}}

\def\STP{\operatorname{ST^+}}
\def\Inn{\operatorname{Inn}}
\def\Aut{\operatorname{Aut}}

\DeclareMathOperator{\res}{\restriction}
\DeclareMathOperator{\ISO}{Iso}
\begin{document}
	\title[]{Minimality of the inner automorphism  group}
	\author[D. Peng, M. Shlossberg]{Dekui Peng, Menachem Shlossberg}

   \address[D. Peng]
	{\hfill\break Institute of Mathematics,
		\hfill\break Nanjing Normal University, 210024,
		\hfill\break China}
\email{pengdk10@lzu.edu.cn}

	\address[M. Shlossberg]
	{\hfill\break School of Computer Science
		\hfill\break Reichman University, 4610101 Herzliya
		\hfill\break Israel}
	\email{menachem.shlossberg@post.runi.ac.il}
	\maketitle	
	\setcounter{tocdepth}{1}
	\begin{abstract} By \cite{DHPXX},
	a minimal group $G$  is called \textit{$z$-minimal}  if $G/Z(G)$ is minimal.   In this paper,  we present the \textit{$z$-Minimality Criterion} for  dense subgroups with some applications to topological matrix groups.  
For  a locally compact  group $G$,  let $\Inn(G)$ be the group of all inner automorphisms of $G,$ endowed with the Birkhoff topology.   Using  a theorem by Goto \cite{Goto73}, we obtain our main result which asserts that if $G$ is a connected Lie group    and $H\in\{G/Z(G), \Inn(G)\},$ then $H$ is minimal if and only if it   is centre-free and topologically isomorphic to  $\Inn(G/Z(G)).$ In particular, if $G$ is a connected Lie group with discrete centre, then  $\Inn(G)$ is minimal.
We prove  that  a  connected locally compact nilpotent group is $z$-minimal if and only if it is compact abelian. In contrast, we show that there exists a connected  metabelian $z$-minimal Lie group that is neither compact nor abelian.
\end{abstract}
\tableofcontents

\section{Introduction}

All topological groups in this paper are Hausdorff.
 Recall that a
topological group $G$ is  \emph{minimal} \cite{Doitch,S71} if
 it does not admit a strictly coarser Hausdorff group topology.   By Prodanov--Stoyanov \cite{PS},  minimal abelian groups  are  precompact. If  every quotient of $G$ is minimal, then $G$ is called \emph{totally minimal} \cite{DP}. A totally minimal group is minimal while the converse is not true. Indeed,
  by  Dierolf and Schwanengel \cite{DS79} the group
 \begin{align*}
 	\R \rtimes \R_{+} & \cong \left\{\left( \begin{array}{cc}
 		a & b  \\
 		0 & 1
 	\end{array}  \right) \middle|  \ a \in \R_{+}, \  b \in \R \right\}.
 \end{align*}
is minimal but not totally minimal.

Minimality is not preserved under taking closed subgroups while a closed central subgroup of a minimal group is minimal (see, for example, \cite[Proposition 7.2.5]{DPS89}). This  led to the study  of the \emph {$c$-minimal} groups in \cite{XS}.
Namely, the topological groups having all closed subgroups minimal.
  For more information about minimality we refer the reader to the survey paper \cite{DM14}.

 Recently,  
 the authors of \cite{DHPXX} introduced  the following concept which we continue to explore in this paper.
  \begin{definition}
 	A minimal group $G$  is $z$-minimal  if $G/Z(G)$ is minimal.
 \end{definition}
Clearly, for every centre-free topological group $G$ we have the topological isomorphism $G\cong G/Z(G).$ So,  a centre-free topological group is $z$-minimal if and only if it is minimal.

The product of finitely many complete minimal groups is minimal (see, for example, \cite[Corollary 7.3]{DPS89}). Uspenskij \cite{U} asked whether the class of complete minimal groups is closed to  infinite products. Answering this question partially, Megrelishvili \cite{MEG95} proved that the product of minimal groups with finite centre is minimal.  The authors of \cite{DHPXX} extended this result to a family of  $z$-minimal groups $\{G_i: i\in I\}$ for which the product  $\prod_{i\in I} Z(G_i)$ is minimal. In particular, the product of complete $z$-minimal groups is minimal.

In  Theorem  \ref{thm:zcrit}, we present the \textit{$z$-Minimality Criterion} for  dense subgroups. 	As in the papers \cite{MS,SH}, some applications to Number Theory are provided (see  Proposition \ref{prop:sfree} and Theorem \ref{thm:minprod}).

For a  a locally compact group  $G$ one can consider two group topologies on   $\Inn(G).$
One is the topology induced by the quotient topology on $G/Z(G)$ and the other is the Birkhoff topology which is always coarser. 
 Among the  many researchers who studied the minimality of Lie groups one can find
 van Est \cite{Est}, Omori \cite{Omori}, Goto \cite{Goto73} and  Remus--Stoyanov \cite{RS}.  
As our main result, we prove in Theorem \ref{thm:curmain} and Theorem \ref{thm: minofgzg} that if  $G$ is a connected Lie group    and $H\in\{G/Z(G), \Inn(G)\},$ then $H$ is minimal if and only if it   is centre-free and topologically isomorphic to  $\Inn(G/Z(G)).$
In particular, if $G$ is a connected Lie group with discrete centre (e.g., when $G$ is a semi-simple connected Lie group), then  $\Inn(G)$ is minimal (see Theorem \ref{Th:Min}).
By corollary \ref{cor:ofnilp},  a  connected locally compact nilpotent group is $z$-minimal if and only if it is compact abelian. In contrast, we show in Example \ref{ex:solv}  that there exists a connected  metabelian $z$-minimal Lie group that is neither compact nor abelian.

\subsection{Notation and terminology}
The fields of rationals, reals and complex numbers   are denoted  by $\mathbb{Q},\mathbb{R} $ and  $\mathbb{C}$, respectively, and $\T=\{a\in \C\ : \ |a|=1\}$ denotes the unit circle group.
For a prime number $p$,  $\Q_p$ is the field of $p$-adic numbers equipped with the $p$-adic topology $\tau_p.$ We denote by $\N$ and $\R_{+}$ the set of positive natural numbers and positive reals, respectively.
%

Let $G$ be a group and $x\in G.$
 We denote by  $\langle x\rangle$ the subgroup of $G$ generated by $x.$
The group $G$ is \emph{solvable} if  there exist $k\in \N$ and a subnormal series $$G_0=\{e\}\unlhd G_1\unlhd \cdots \unlhd G_k=G,$$  where $e$ denotes the identity element, such that the quotient group $G_j/G_{j-1}$  is abelian for every $j\in \{1,\ldots, k\}.$ In particular, $G$ is  {\em metabelian}, if $k\leq 2.$
The subgroup $Z(G)$ denotes the \emph{centre} of $G,$ and we set $Z_0(G) = \{e\}$ and  $Z_1(G) = Z(G)$.  For $n > 1$, the  \emph{$n$-th centre} $Z_n(G)$ is defined as follows:
$$Z_n(G)=\{x\in G: [x,y]\in Z_{n-1}(G)  \text{ for every } y\in G\},$$
where $[x,y]$ denotes the commutator $xyx^{-1}y^{-1}$.  A group is {\em nilpotent} if $Z_n(G) = G$ for some $n\in \N$. 
The {\em derived subgroup} of $G$ is generated by all commutators $[x,y]$, where $x,y\in G$.

 A  topological group is called  \emph{complete} if it is complete with respect to its two-sided uniformity, and it is \emph{precompact} if its completion is compact. 
 
For a locally compact group $G$, we denote by $\Aut(G)$ and $\Inn(G)$ the group of all automorphism of $G$ and its subgroup of all inner automorphisms, respectively.  It is known that algebraically $G/Z(G)$ is isomorphic to $\Inn(G)$. This isomorphism is given by sending $gZ(G)$ to $I_g,$ the inner automorphism induced by $g\in G.$
\section{The $z$-Minimality Criterion}
 For a subgroup  $H$ of a topological group $(G,\gamma)$ denote by $\gamma|_{H}$ and  $\gamma/H$ the subspace topology and the quotient topology induced by $\gamma$ on $H$ and $G/H$, respectively. The following result is known as \textit{Merson's Lemma} (see \cite[Lemma 7.2.3]{DPS89} or \cite[Lemma 4.4]{DM14}).
 \begin{fact}\label{firMer}
 	Let $(G,\gamma)$ be a (not necessarily Hausdorff) topological group and $H$ be a subgroup of $G.$ If $\gamma_1\subseteq \gamma$ is a coarser group topology on $G$ such that
 	$\gamma_1|_{H}=\gamma|_{H}$ and $\gamma_1/H=\gamma/H$, then $\gamma_1=\gamma.$
 \end{fact}
Recall that a subset $H$ of a group $G$ is called {\it unconditionally closed} if $H$ is closed in every Hausdorff group topology of $G$ (see \cite{Markov}).
It is   known that the $n$-centres are unconditionally closed.
 \begin{lemma}\label{lem:zminiff}A topological group $G$ is $z$-minimal if and only if both $Z(G)$ and $G/Z(G)$ are minimal.\end{lemma}
 \begin{proof}First we assume that $G$ is $z$-minimal. Since minimality is preserved by taking closed central subgroups, $Z(G)$ is minimal.
The minimality of $G/Z(G)$ follows from the definition of $z$-minimal groups.

For the other direction, it suffices to show that $G$ is minimal under the assumption that $Z(G)$ and $G/Z(G)$ are minimal.
This is an immediate corollary of Merson's Lemma (Fact \ref{firMer}) 
as $Z(G)$ is unconditionally closed.  \end{proof}
%

\begin{corollary}
If $G$ is a $z$-minimal group, then  $Z_2(G)$ is precompact. If $G$ is also complete, then $Z_2(G)$ is compact.
\end{corollary}
\begin{proof}
Both $Z(G)$ and  $Z(G/Z(G))=Z_2(G)/Z(G)$ are precompact, being minimal abelian.   As precompactness is a three space property (see \cite{BT} for example), it follows that  $Z_2(G)$ is precompact. If $G$ is also complete, then its unconditionally closed subgroup
 $Z_2(G)$ must be also complete, hence, compact.
\end{proof}


\begin{definition}
	Let $H$ be a subgroup of a topological group $G$. Then $H$ is   essential in $G$ if $H\cap L \neq \{e\}$
	for every non-trivial closed normal subgroup $L$ of $G$.
\end{definition}
The following Minimality Criterion of dense subgroups is well-known (for compact $G$ see also \cite{P,S71}).
\begin{fact} \label{Crit} \cite[Minimality Criterion]{B}
	Let $H$ be a dense subgroup of a topological group $G.$  Then $H$ is minimal if and only if $G$ is minimal and $H$ is essential in $G$.
\end{fact}
We present here a similar criterion for the $z$-minimality of dense subgroups. This answers a question of G. Luk\' acs.
\begin{theorem}[$z$-Minimality Criterion]\label{thm:zcrit}Let $H$ be a dense subgroup of a topological group $G$. Then $H$ is $z$-minimal if and only if $G$ is $z$-minimal and
	\begin{itemize}
		\item[(a)] $Z(H)$ is a dense essential subgroup of $Z(G)$; and
		\item[(b)] every closed normal subgroup of $G$ properly containing $Z(G)$ contains a non-central element of $H$.
	\end{itemize}
\end{theorem}
\begin{proof}
	We first assume that $H$ is $z$-minimal.
Then, in particular, $H$ is minimal.
	By  the denseness of $H$ in $G$ and the Minimality Criterion, 
		we get that $G$ is minimal and $Z(H)=Z(G)\cap H$.
	Let $q: G\to G/Z(G)$ be the canonical mapping. Then the kernel of $q\res_H$ is $Z(G)\cap H=Z(H)$.
	The minimality of $H/Z(H)$ implies that the restriction mapping $q\res_H: H\to q(H)$ is open, in particular, $q(H)$ is topologically isomorphic to $H/Z(H)$.
	According to \cite[4.3.2 Lemma]{DPS89}, $Z(H)$ must be dense in $Z(G)$.
	Since $q(H)\cong H/Z(H)$ is a dense minimal subgroup of $G/Z(G)$,  
	the Minimality Criterion   implies that $G/Z(G)$ is minimal (so $G$ is $z$-minimal) and $q(H)$ is essential in $G/Z(G)$.
	The latter implies (b) since for every closed normal subgroup $N$ of $G$ properly containing $Z(G)$, $q(G)$ is non-trivial and hence $q(N)\cap q(H)\neq \{e\}$, or equivalently, $N\cap H\nsubseteq Z(G)$.
	It remains to note that $Z(H)\subseteq Z(G)$, so $N\cap H\nsubseteq Z(H)$.
	To see (a), note that minimality is preserved by taking centres (so $Z(H)$ is minimal) and we have already proved that  $Z(H)$ is dense in $Z(G)$, thus the Minimality Criterion applies.
	
	Now we assume that $G$ is $z$-minimal and (a) and (b) are satisfied.
	Again the centre $Z(G)$ of the minimal group $G$ is minimal.
	By (a) and the Minimality Criterion, 
	$Z(H)$ is minimal.
	Moreover, as $Z(H)$ is dense in $Z(G)$, the image $q(H)$ of $H$ under the canonical mapping $q:G\to G/Z(G)$ is topologically isomorphic to $$H/(H\cap Z(G))=H/Z(H).$$
	It remains to see that $q(H)$ is a minimal group. By the Minimality Criterion, it suffices to show that $q(H)$ is essential in $G/Z(G)$.
	Let $M$ be a closed normal subgroup of $G/Z(G)$ distinct from the trivial group.
	Then $q^{-1}(M)$ is a closed normal subgroup of $G$ properly containing $Z(G)$.
	By (b), there exists $x\in q^{-1}(M)\cap H$ such that $x\notin Z(H)$.
	As $Z(H)=Z(G)\cap H$ and $x\in H$, this implies that $x$ does not belong to $Z(G)$,
	thus, $e\neq q(x)\in M\cap q(H)$.
	In other words, $q(H)$ is essential in $G/Z(G)$ and $H/Z(H)$ is minimal.
	Now apply Lemma \ref{lem:zminiff}.
\end{proof}

\subsection{$z$-Minimality of $\STP(n,\F)$ and $\SL(n,\F)$}
Let $\F$ be a topological field.
Denote by $\SL(n,\F)$  the special linear group  over  $\F$ of degree $n$  equipped with the 
pointwise topology
inherited from $\F^{n^2},$ and by $\STP(n,\F)$ its topological subgroup  consisting of upper triangular matrices. Let $\mathsf{N}:=\UT(n,\F)$ and $\mathsf{A}$ be the subgroups of $\STP(n,\F)$     consisting of upper unitriangular matrices and diagonal matrices, respectively.  Note that $\mathsf{N}$ is normal in $\STP(n,\F)$ and  $\STP(n,\F)\cong \mathsf{N} \rtimes_{\alpha} \mathsf{A}$, where  $\alpha$ is the action by conjugations. It is known that  $\mathsf{N}$ is the derived subgroup of  $\STP(n,\F).$
Recall also that $\SL(n,\F)$ has finite center (e.g., see \cite[3.2.6]{RO} or \cite[page 78]{SU})
\[Z=Z(\SL(n,\F))=\{\lambda I_n:\lambda\in \mu_n\},\] where $\mu_n$ is a finite cyclic group  consisting of the $n$-th roots of unity in $\F$ and $I_n$ is the identity matrix of size $n.$ Moreover, $Z(\STP(n,\F))=Z$ and  $\STP(n,\F)/Z$ is  centre-free (see  \cite[Lemma 1]{SH}).
\vskip 0.3 cm
Clearly, if $G$ is totally minimal, then it is $z$-minimal.
The following lemma provides a case in which also the converse is true.
\begin{lemma}\label{lem:zimptot} Let $\F$ be a topological field. Then
	$\SL(n,\F)$ is   totally minimal if and (only if) it is $z$-minimal.
\end{lemma}
\begin{proof}
	Being algebraically simple (see \cite[3.2.9]{RO}), the projective special linear group  $$\PSL(n,\F)= \SL(n,\F)/Z$$ is totally minimal if $\SL(n,\F)$ is  $z$-minimal. Since $Z$ is a finite normal subgroup of $\SL(n,\F)$ it follows from \cite[Theorem 7.3.1]{DPS89} that
	$\SL(n,\F)$ is   totally minimal.
\end{proof}
Recall that  a locally compact non-discrete field is called a \textit{local field}.
\begin{remark}\label{remark:forlocal}
	If $\F$ is a subfield of a local field $P,$ then its
	completion $\widehat\F$  is a topological field that can be identified with the closure of $\F$ in $P$. In case $\F$ is  infinite then  $\widehat\F$ is also a local field, as
	the local field $P$ contains no infinite discrete subfields (see \cite[p.  27]{MA}).
\end{remark}
The following result is an immediate corollary of Theorem 3.17, Theorem 3.19 and Theorem 4.3 of  \cite{MS}.
\begin{corollary}\label{cor:zminmat}
	Let $\F$ be a  local field and $n\in \N.$
	Then \begin{enumerate}
		\item $\SL(n,\F)$ is $z$-minimal (equivalently, totally minimal);
		\item if $\F$ has a characteristic distinct from $2,$ then $\STP(n,\F)$ is $z$-minimal.
	\end{enumerate}
\end{corollary}

\begin{proposition}\label{prop:zminifmin}
	Let $\F$ be a  subfield of a local field of characteristic distinct from $2.$  Then $\SL(n,\F)$ is   $z$-minimal if and only if $\STP(n,\F)$ is $z$-minimal.
\end{proposition}
\begin{proof} Clearly, we may assume that $\F$ is infinite and $n\geq 2.$ Assume first that $\STP(n,\F)$ is $z$-minimal and let us show that $\SL(n,\F)$ is   totally minimal. Note that to prove this  implication we do not need to assume anything about the characteristic of $\F.$
	By \cite[Theorem 4.7]{MS}, it suffices to show that $Z(\SL(n,\F))=Z(\SL(n,\widehat\F)).$
	
	In view of the 	$z$-minimality of $\STP(n,\F)$ and
	item (a) of Theorem \ref{thm:zcrit}, we deduce that
	$Z(\STP(n,\F))=Z(\STP(n,\widehat\F))$ since these centres are finite. Moreover, we always have $Z(\STP(n,\F))=Z(\SL(n,\F))$ and $Z(\STP(n,\widehat\F))=Z(\SL(n,\widehat\F))$ so it follows that $Z(\SL(n,\F))=Z(\SL(n,\widehat\F)),$  as needed.
	
	To prove the converse implication, we assume, in addition, that the characteristic of $\F$ is distinct from $2.$ By Corollary \ref{cor:zminmat}, $\STP(n,\widehat\F)$ is $z$-minimal.
	By Lemma \ref{lem:zimptot}, $\SL(n,\F)$ is    totally minimal. Then $Z=Z(\SL(n,\F))=Z(\SL(n,\widehat\F))$ in view of \cite[Theorem 4.7]{MS}. So, we also have $Z=Z(\STP(n,\F))=Z(\STP(n,\widehat\F)).$  Let us assume first that $n\geq 3.$ Observe that by the Minimality Criterion, it suffices to prove that  the dense subgroup $\STP(n,\F)/Z$  is essential in the minimal group $\STP(n,\widehat\F)/Z$.  To this aim, let $L$ be a non-trivial closed normal subgroup of $\STP(n,\widehat\F)/Z.$ Clearly, this means that $q^{-1}(L)$ is a normal subgroup of $\STP(n,\widehat\F)$ which is not central, where   $$q:\STP(n,\widehat\F)\to \STP(n,\widehat\F)/Z$$ is the quotient map. Then, since $n\geq 3$ the normal subgroup $q^{-1}(L)$ must non-trivially intersect  $\UT(n,\widehat{\F}),$  the derived subgroup of $\STP(n,\widehat{\F})$, in view of  \cite[Lemma 2.3]{XDSD}. By \cite[Lemma 2]{SH},  $q^{-1}(L)$ intersects  $\UT(n,\F)$ non-trivially. This implies that  $L$ intersects $\STP(n,\F)/Z$ non-trivially.
	
	Now we prove the  case $n=2$. By Corollary \ref{cor:zminmat}.2, $G:=\STP(2,\widehat\F)$ is $z$-minimal. Hence, to see that $H:=\STP(2, \F)$ is $z$-minimal,  we need  to verify the validity of conditions (a) and (b) of Theorem \ref{thm:zcrit}.
	
	
	
	Condition (a) is satisfied as   $Z(H)=Z(G)=\{\pm I\}.$
	According to the argument in the proof of \cite[Theorem 3.4]{MS}, every   normal subgroup $L$  of $G$  properly containing $Z(G)$ contains  the element $\left(\begin{array}{cc}
		1 & 1 \\
		0 & 1 \\
	\end{array}
	\right);$
	thus (b) is satisfied.
\end{proof}
Recall that $n\in \N$ is \textit{square-free} if $n$ is divisible by no square number other than $1.$
\begin{corollary} \label{cor:sq} Let $n\in \N$  and $\mathbb{F}$ a subfield of a local field.
	If $n$ is 
	square-free,
	then $\SL(n, \mathbb{F})$ is minimal if and only if it is totally minimal. If, in addition, $\F$ has a characteristic distinct from $2,$ then $\STP(n,\F)$ is minimal if and only if it is $z$-minimal.
\end{corollary}
\begin{proof}
	By \cite[Proposition 5.1]{MS}, if $H:=\SL(n, \mathbb{F})$ is minimal, then $Z(H)$ meets non-trivially any non-trivial central subgroup of $G:=\SL(n, \widehat{\mathbb{F}})$.
	On the other hand, $Z(G)$ is isomorphic to the group of solutions of the equation $x^n=1$ in $\widehat{\mathbb{F}}$.
		%
	Since $n$ is square-free, this  centre is a cyclic group of order a product of distinct primes. It follows
	that any subgroup of $Z(G)$ with the announced property can only be $Z(G)$ itself.
	Apply \cite[Theorem 4.7]{MS} to deduce that $H$ is totally minimal. By \cite[Theorem 2]{SH} and Proposition \ref{prop:zminifmin},  if $\F$ has a characteristic distinct from $2,$ then  $\SL(n,\F)$ is  minimal ($z$-minimal)  if and only if $\STP(n,\F)$ is minimal (resp., $z$-minimal). This proves the last assertion.
\end{proof}
\begin{proposition} \label{prop:sfree} Let $n$ be a positive integer. Then the following conditions are equivalent:
	\begin{itemize}
		\item[(a)] $n$ is not square-free;
		\item[(b)] there exists a subfield $\mathbb{F}$ of $\mathbb{C}$ such that $\SL(n, \mathbb{F})$ is minimal but not totally minimal;
		\item [(c)] there exists a subfield $\mathbb{F}$ of $\mathbb{C}$ such that $\STP(n, \mathbb{F})$ is minimal but not $z$-minimal.
	\end{itemize}
\end{proposition}

\begin{proof}
	$(b)\Rightarrow (a)$ Use Corollary \ref{cor:sq}.\\
	$(a)\Rightarrow (b)$
	Now assume that there exists a prime number $p$ with $p^2\mid n$.
	Let $n=pm$, $\alpha=e^{\frac{2\pi i}{n}}$ and $N=\hull{\alpha}$.
	Consider the subfield $\mathbb{F}=\mathbb{Q}(\alpha^p)$ of $\mathbb{C}$.
	Then $N':=\mathbb{F}\cap N=\hull{\alpha^p}$.
	Since the subgroup $N'$ contains the socle of the finite cyclic group $N$, it must non-trivially meet any non-trivial subgroup of $N$.
	By \cite[Proposition 2]{SH}, $\SL(n, \mathbb{F})$ is minimal but not totally minimal.\\
	$(b)\Leftrightarrow(c)$ Use \cite[Theorem 2]{SH}, Proposition \ref{prop:zminifmin} and Lemma \ref{lem:zimptot}, taking into account that the field $\C$ has zero characteristic.
\end{proof}

The next theorem answers \cite[Question 6]{SH} in the positive.
\begin{theorem}\label{thm:minprod}
	Let $\mathcal P$ be the set of all primes. Then  $G=\prod_{p\in \mathcal P}\SL(p+1,(\Q,\tau_p))$ and
	$H=\prod_{p\in \mathcal  P}\STP(p+1,(\Q,\tau_p))$ are both minimal.
\end{theorem}
\begin{proof}
	For every $p\in\mathcal P$ we have $Z(\SL(p+1,\Q_p))=Z(\STP(p+1,\Q_p))\subseteq \{I,-I\}.$ By \cite[Corollary 4.8]{MS}, $\SL(p+1,(\Q,\tau_p))$ is totally minimal. Since $Z(G)$ is compact we deduce from \cite[theorem 4.8]{DHPXX} that $G$ is minimal. By Proposition \ref{prop:zminifmin}, $\STP((p+1,(\Q,\tau_p))$ is $z$-minimal for every prime $p.$ Using \cite[Theorem 4.8]{DHPXX} again and the fact that $Z(H)=Z(G)$, we establish the minimality of	$H$.
\end{proof}

\section{The Birkhoff topology}

Let $H$ be a group of topological automorphisms of a locally compact group $G$. 
Equip $H$ with a topology $\tau$.
It is known 
that the action of $H$ on $G$ is continuous if and only if $\tau$ is finer than the compact-open topology.
However, in general the compact-open topology is not a group topology on $H$.
So, we may consider whether there is a coarsest group topology on $H$ making this action continuous, i.e., 
finer than the compact-open topology.
Fortunately, such a topology exists.
It is called the \emph{Birkhoff topology} (see \cite[p. 260]{DPS89} and \cite[Remark 1.9]{MEG95}) and also the \emph{generalized compact-open topology}.
 This group topology has a local base at the identity formed by the sets
\[\mathcal B(C,U):=\{f\in H| \ f(g)\in Ug,  \ f^{-1}(g)\in Ug  \ \forall g\in C\},\] where
$C$ runs over compact subsets and $U$ runs over neighbourhoods of the identity in $G.$
So, if the compact-open topology makes $H$  a topological group, then the two topologies coincide. 
 In the sequel, a subgroup of $\Aut(G)$ will always carry the Birkhoff topology. When $G$ is a connected Lie group, this topology coincides with the compact-open topology, and makes $\Aut(G)$ 
a Lie group \cite[Theorem 1]{Hoc}.

There is another 
way to obtain this Lie group topology, see \cite{Hoc, HN}.
Let $\varphi: G\to H$ be a quotient homomorphism of connected Lie groups.
We say that $\varphi$ is a {\em covering} if $\ker \varphi$ is discrete. In this case, $G$ is a {\em covering group} of $H.$ If, in  addition, $G$ is simply connected, then it is the {\em universal covering group} of $H$.
The universal covering group is unique, up to a topological isomorphism. We 
 denote the universal covering group of $G$ by $\widetilde{G}$.  Moreover, 
connected Lie groups have isomorphic universal covering groups  if and only if they have isomorphic Lie algebras \cite[Theorem 9.5.13]{HN}.  In particular, if $G$ is a covering group of $H$, then they have isomorphic Lie algebras.

Let $\mathfrak{g}$ be the Lie algebra of $G$, then $\Aut(\mathfrak{g})$, the group of Lie algebra automorphisms, is a closed subgroup of the general linear group $\GL(\mathfrak{g})\cong \GL(n, \mathbb{R})$, where $n$ is the dimension of $G$.
It is also known that, $\Aut(\widetilde{G})$ and $\Aut(\mathfrak{g})$ are isomorphic as groups (see \cite[Corollary 9.5.11]{HN}). 
Every (topological) automorphism $\varphi$ of $G$ admits a ``lifting'' to a unique automorphism $\widetilde{\varphi}$ of $\widetilde{G}$ \cite[Remark 9.55(b)]{HN}. That is, $\widetilde{\varphi}$ makes the following diagram commutative

$$\xymatrix{
  & \widetilde{G} \ar[d]_{q} \ar[r]^{\widetilde{\varphi}}  & \widetilde{G} \ar[d]^q            \\
  & G \ar[r]^{\varphi} & G                       }
 $$
 where $q$ is the covering mapping. So,
the mapping $\varphi \mapsto \widetilde{\varphi}$ gives a group embedding of $\Aut(G)$ into $\Aut(\widetilde{G})$, with image the setwise stabilizer of $N$, where $N$ is the kernel of the covering $\widetilde{G}\to G$.  Since the latter group is isomorphic to $\Aut(\mathfrak{g})$, a Lie subgroup of the linear group $\GL(\mathfrak{g})$, we obtain a Lie group structure on $\Aut(G)$. This topology also coincides with the Birkhoff topology \cite[IX. Theorem 1.2]{Hoch}.

The topological group embedding $\varphi\mapsto \widetilde{\varphi}$  maps $\Inn(G)$ onto $\Inn(\widetilde{G})$.
So, $\Inn(G)$ and $\Inn(\widetilde{G})$ are isomorphic as topological groups.
This implies that connected Lie groups with isomorphic Lie algebras have topologically isomorphic inner automorphism groups.

A locally compact group $G$ is called a {\em (CA) group} if $\Inn(G)$ is closed in $\Aut(G)$. 
If $G$ is a connected Lie group, 
then $\Aut(G)$ is a Lie group by the above argument.
So $\Inn(G)$ is closed in $\Aut(G)$ if and only if $\Inn(G)$ is locally compact.
Therefore, the (CA) property is totally determined by Lie algebras.
In summary, we have:

\begin{fact} \label{InnIso}
Let $G,H$ be connected Lie groups with isomorphic Lie algebras, then $\Inn(G)\cong \Inn(H)$.
In particular, $G$ is (CA) if and only if $H$ is (CA).
\end{fact}
Let $G$ be a locally compact group, then the natural homomorphism $$G\to \Inn(G), g\mapsto I_g$$
is continuous.
Indeed, since the kernel of this mapping is $Z(G)$, one may only check the induced group isomorphism $G/Z(G)\to \Inn(G)$ is continuous.
Clearly, the action $\alpha: (G/Z(G),\tau_q)\times G\to G$ defined by $\alpha(gZ(G),h)=ghg^{-1}$ is continuous, where $\tau_q$ is the quotient topology on $G/Z(G)$.  
Since the Birkhoff topology $\tau_B$ is the coarsest group topology making the natural action of $\Inn(G)$ on $G$ continuous,
we get that $\tau_B\subseteq \tau_q$, which shows that 
 the mapping is continuous.
So, we have the following:

\begin{proposition}\label{prop:birk}
Let $G$ be a locally compact group. If the quotient group $G/Z(G)$ is minimal, then it is topologically isomorphic to
$\Inn(G).$
In particular, if $G$ is a locally compact $z$-minimal group, then $\Inn(G)$ is locally compact minimal. 
\end{proposition}

The Lie algebras of (CA) connected Lie groups are called {\em (CA) Lie algebras}. They were studied also by van Est \cite{EstAlg}.
 The following topological groups are all (CA).
\begin{example}\label{ExCA} 
\begin{itemize}
  \item[(1)] Compact groups and more generally 
    locally compact $z$-minimal groups in view of Proposition \ref{prop:birk};
  \item[(2)]\cite[Theorem 3.3]{Zer} every connected Lie group of dimension $\leq 4$;
  \item[(3)]\cite{EstAlg}   connected nilpotent Lie groups;
  \item[(4)]\cite{EstAlg}  semi-simple connected Lie groups. This is because a connected Lie group is (CA) if and only if its radical (the largest closed solvable connected normal subgroup) is (CA).
\end{itemize}
\end{example}
\begin{proposition}\label{CAiff}
Let $G$ be a connected Lie group, then $G$ is (CA) if and only if $\Inn(G)\cong G/Z(G)$.
\end{proposition}
\begin{proof}
The sufficiency is evident.
To see the converse, suppose that $G$ is (CA).
Then $\Inn(G)$ is a locally compact group.
Since $G$ is connected,
we deduce by \cite[Theorem 3.1.27]{AT}, in view of the  Baire Category Theorem,  that the natural mapping $G\to \Inn(G)$ is open.
So $\Inn(G)\cong G/Z(G)$.
\end{proof}

\begin{proposition}\label{InnProd}
Let $G$, $H$ be locally compact groups.
Then $\Inn(G\times H)\cong \Inn(G)\times \Inn(H)$.
\end{proposition}
\begin{proof}
Clearly, the mapping 
$$\iota: \Inn(G\times H)\to \Inn(G)\times \Inn(H), I_{(g,h)}\mapsto (I_g, I_h)$$
 is an isomorphism of abstract groups.
Now let $K_1, K_2$ be compact subsets of $G$ and $H$, and $U_1, U_2$ be identity neighbourhoods in $G$ and $H$, respectively.
To see that $\iota$ is continuous, it suffices to check that the preimage of 
$\mathcal{B}(K_1, U_1)\times \mathcal{B}(K_2, U_2)$
is open in $\Inn(G\times H)$.
While, this follows from that 
$$\iota^{-1}( \mathcal{B}(K_1, U_1)\times \mathcal{B}(K_2, U_2))=\mathcal{B}(K_1\times K_2, U_1\times U_2)$$
immediately.

Now let us see that $\iota$ is also open.
Let $K\subseteq G\times H$ be compact and $U$ be a neighbourhood of the identity in $G\times H$.
Choose compact subsets $K_1, K_2$ of $G$ and $H$ respectively such that $K\subseteq K_1\times K_2$.
Also take identity neighbourhoods $U_1$ and $U_2$ of $G$ and $H$ respectively with $U_1\times U_2\subseteq U$.
Then we get that $\mathcal{B}(K_1\times K_2, U_1\times U_2)\subseteq \mathcal{B}(K, U)$ and hence
$$\mathcal{B}(K_1, U_1)\times \mathcal{B}(K_2, U_2)=\iota(\mathcal{B}(K_1\times K_2, U_1\times U_2))\subseteq \iota(\mathcal{B}(K, U)).$$
This proves that $\iota$ is open.
\end{proof}

\begin{corollary}
Let $G_1, G_2,...,G_n$ be connected Lie groups. Then $\prod_{i=1}^n G_i$ is (CA) if and only if each $G_i$ is (CA).
\end{corollary}
\begin{proof}
We need only to consider the case $n=2$.
If both $G_1$ and $G_2$ are (CA), then by Proposition \ref{CAiff} and Proposition \ref{InnProd}, 
$$\Inn(G_1\times G_2)\cong \Inn(G_1)\times \Inn(G_2)\cong G_1/Z(G_1)\times G_2/Z(G_2)\cong (G_1\times G_2)/Z(G_1\times G_2).$$
Using Proposition \ref{CAiff} again, we obtain that $G_1\times G_2$ is (CA).

Conversely, if $G_1\times G_2$ is (CA).
Then $$\Inn(G_1)\times \Inn(G_2)\cong \Inn(G_1\times G_2)\cong (G_1\times G_2)/Z(G_1\times G_2)$$ is locally compact.
As a closed subgroup, $\Inn(G_i)$ is also locally compact, so closed in $\Aut(G_i)$, where $i=1,2$.
That is, $G_1$ and $G_2$ are (CA).
\end{proof}

There is a deep connection 
between the (CA) property and the minimality of a connected Lie group as the following theorem by Goto \cite[Theorem]{Goto73} shows.

\begin{fact}
 A connected Lie group is minimal if and only if it is a (CA) group with  compact centre.
\end{fact}

Combining Proposition \ref{CAiff} with Goto's theorem, 
we have:
\begin{proposition}\label{prop:inisq}
Let $G$ be a connected  Lie group.  Then   $G$ is minimal if and only if $Z(G)$ is compact and $\Inn(G)\cong G/Z(G).$
\end{proposition}
Motivated by the definition of $z$-minimality we present the following concept.
\begin{definition}
A  locally compact minimal group $G$ is  \textit{$\Inn$-minimal} if $\Inn(G)$ is minimal.
\end{definition}

Locally compact $z$-minimal groups are necessarily $\Inn$-minimal. By the isomorphism  $\Inn(G)\cong G/Z(G),$ given in 
Proposition \ref{prop:inisq}, we obtain the following corollary.
\begin{corollary}\label{cor:zifinn}
Let $G$ be a connected  Lie group. Then $G$ is $\Inn$-minimal if and only if it is $z$-minimal.
\end{corollary}

The following fact 
	is well-known in 
	 case 
	 $G$ is a Lie group and $N$ is discrete. For 
	 the reader sake we provide a proof for 
	 a more general case.
\begin{fact}
\label{centre}
Let $G$ be a connected topological group and $N$be  a closed normal subgroup. If $N$ is totally disconnected, and $q: G\to G/N$ is the canonical quotient mapping, then
\begin{itemize}
  \item[(1)] $N$ is central;
  \item[(2)] $q(Z(G))=Z(G/N)$ and $q^{-1}(Z(G/N))=Z(G)$.
\end{itemize}
\end{fact}
\begin{proof}
(1) follows from the second equation in (2).

Now let us check (2).
 The first equation is also a consequence of the second.
Fix $x\in q^{-1}(Z(G/N))$, then for any $g\in G$, $q(x)$ commutes with $q(g)$.
Hence we have that $gxg^{-1}x^{-1}\in \ker q=N$.
Thus we obtain a continuous mapping 
$$\psi_x: G\to N, g\mapsto gxg^{-1}x^{-1}.$$
As $G$ is connected while $N$ is totally disconnected, $\psi_x(G)$ must be a singleton.
Since $1=\psi_x(x)\in \psi_x(G)$, we have that $\psi_x(G)$ is exactly the trivial group.
Thus, $x\in Z(G)$.
So we have that $q^{-1}(Z(G/N))\subseteq Z(G)$.
The inclusion in the other direction is evident.
\end{proof}

\begin{proposition}\cite[Proposition 8]{Goto55}\label{Prop:Goto55}
Let $G$ and $H$ be connected locally compact groups, and let $\varphi: G\to H$ be a continuous monomorphism with dense image. Then
\begin{itemize}
  \item[(1)] $\varphi(G)$ is normal in 
  $H.$
  \item[(2)] The action $\sigma: H\times G\to G$ defined by $\sigma(h,g):=\varphi^{-1}(h^{-1}\varphi(g)h)$; $h\in H, g\in G$ is continuous.
\end{itemize}
\end{proposition}

\begin{remark}\label{Re:action}
Let $N$ be the kernel of the action in Proposition \ref{Prop:Goto55}(2). Then the induced faithful action of $H/N$ on $G$ is continuous.
Recall that the coarsest group topology on a subgroup of $\Aut(G)$ making the action on $G$ continuous    is the Birkhoff topology.
So the natural injective mapping $H/N\to \Aut(G)$ is continuous.
Thus, we obtain a continuous homomorphism $\sigma'$ of $H$ into $\Aut(G)$ mapping $h$ to $\sigma(h,-)$.
\end{remark}
\begin{theorem}\cite[Main Structure Theorem]{Zer}\label{nonCA}
Let $G$ be a connected Lie group which is not (CA). Then there exist a (CA) connected Lie group $M$, a torus $T$
 in $\Aut(M)$, a dense subgroup $V$ of $T$, and a continuous homomorphism
$\varphi: G\to M\rtimes T$ such that
\begin{itemize}
  \item [(i)] $H:=M\rtimes T$ 
%
is (CA);
  \item [(ii)] $\varphi$ is injective and its image is $M\rtimes V$;
  \item [(iii)] $\varphi(Z(G))$ is contained in $M$ and $\pi(Z(H))$ is finite, where $\pi$ is the projection of $H$ onto $T$.
\item[(iv)]
\cite[Lemma 2.3]{Zer} $\sigma'(H)$ coincides with the closure of $\Inn(G)$ in $\Aut(G)$, where $\sigma'$ is defined as in Remark \ref{Re:action}.
\end{itemize}
\end{theorem}

\begin{proposition}\label{InnMin}
Let $G$ be a centre-free connected Lie group. Then $\Inn(G)$ is minimal.
\end{proposition}

\begin{proof}
If $G$ is (CA), then it is minimal by Goto's theorem. Moreover, $G\to \Inn(G)$ is bijective.
So $\Inn(G)\cong G$ is minimal.

Now assume that $G$ is not (CA).
Let $H, \varphi, M, T, V$ be as in Theorem \ref{nonCA}.
Since $G$ is centre-free and $\varphi$ is injective it follows that $\varphi(G)$ is centre-free. As $M\subseteq \varphi(G),$  one obtains that $Z(H)\cap M$ is trivial.
So by Theorem \ref{nonCA} (iii), $Z(H)$ is finite.

Let $\psi: H\to H/Z(H)$  be the canonical mapping, which is a covering, $H'=\psi(H)$ and $M'=\psi(M)$ (here we consider $M$ as a closed subgroup of $H$ in the natural way).
Then $H'$ is centre-free by Fact \ref{centre}, and (CA) since $H$ and $H'$ have isomorphic Lie algebras, and hence is minimal.
Let $G'=\psi(\varphi(G))$. Then $G'$ is a dense subgroup of $H'$ containing $M'$.
We first show that $G'$ is minimal.
By the Minimality Criterion, it suffices to show that $G'\cap N'\neq \{1\}$ for any non-trivial closed normal subgroup $N'$ of $H'$.
Since $M'\subseteq G'$, we only need to consider 
the case when  $N'$ trivially meets $M'$.
Set $N=\psi^{-1}(N')$.
As $M'\cap N'=\{1\}$ and $M\cap Z(H)=\{1\},$  we deduce that $M\cap N=\{1\}$.
Now, suppose that  $(x,t)\in N\trianglelefteq M\rtimes T$. Then for any $s\in T$,
\begin{equation}\label{e1}(s(x), t)=(1,s)(x,t)(1, s^{-1})\in N.\end{equation}
Hence we have
$$(t^{-1}(x^{-1})t^{-1}(s(x)),1)=(t^{-1}(x^{-1}), t^{-1})(s(x), t)=(x, t)^{-1}(s(x), t)\in N.$$
On the other hand, $(t^{-1}(x^{-1})t^{-1}(s(x)),1)$ is an element of $M\times\{1\}$ (which has been already identified with $M$).
So $(t^{-1}(x^{-1})t^{-1}(s(x)),1)\in M\cap N$ can only be the trivial element, that is, 
$$1=t^{-1}(x^{-1})t^{-1}(s(x))=t^{-1}(x^{-1}s(x)).$$
So $s(x)=x$.
In view of Equation (\ref{e1}), we see that $(x,t)$ commutes with $(1,s)$ for arbitrary $s\in T$.
Since $N$ and $M$ normalize each other, for any $(y, 1)\in M$, 
$$(x, t)(y, 1)(x, t)^{-1}(y, 1)^{-1}\in M\cap N=\{1\}.$$
 That is, $(x, t)$ also commutes with every element of $M$.
Thus, $N$ is in the centre of $H=M\rtimes T$, and hence $N'$ is trivial.
This means 
 that the only closed normal subgroup of $H'$ having trivial intersection with $M'$ is the trivial group.
So $G'$ is essential in $H'$ and hence is minimal.

Now $\mu:=\psi\circ \varphi$ is a continuous homomorphism between the two connected Lie groups $G$ and $H'$ with a dense image.
From Proposition \ref{Prop:Goto55} and Remark \ref{Re:action}, we see that the mapping $\sigma^*: H'\to \Aut(G), h\mapsto \sigma(h, -)$ gives a continuous homomorphism, where $\sigma(h, g)=\mu^{-1}(h^{-1}\mu(g)h)$ for any $g\in G$.
Let $h$ be in the kernel of $\sigma^*$.
Then $h$ commutes with any element of $\mu(G)=G'$.
By the denseness of $G'$ in $H'$, we obtain that $h$ is in the centre of $H'$.
while $H'$ is centre-free. So, $h$ is trivial and $\sigma^*$ is injective.
Now, using the minimality of $H'$, one obtains that $\sigma^*$ is  a topological group embedding.
It follows that $\Inn(G)=\sigma^*(G')$ is also minimal.
\end{proof}

\begin{theorem}\label{Th:Min}Let $G$ be a connected Lie group. If $Z(G)$ is discrete, then $\Inn(G)$ is minimal.
\end{theorem}
\begin{proof}
The group $H:=G/Z(G)$ is centre-free, by Fact \ref{centre}. 
So Proposition \ref{InnMin} yields that $\Inn(H)$ is minimal.
Now, by Fact \ref{InnIso},  we have that $\Inn(G)\cong \Inn(H)$ is minimal.
\end{proof}
Recall that a connected Lie group is called {\em semi-simple} if its Lie algebra is semi-simple, i.e., a direct product of (finitely many) simple Lie algebras.
Evidently, the centre of a semi-simple Lie algebra is trivial.
It follows from   \cite[Proposition 9.23]{HM}  that the centre of a semi-simple connected Lie group is discrete.
So we have:
\begin{corollary}\label{semisimple}
For every semi-simple connected Lie group $G$, $\Inn(G)$ is minimal.
\end{corollary}
\begin{remark}
By \cite{RS},  a connected semi-simple 
Lie group  is totally minimal if and only if its centre is finite. In particular, if $G$ is a a connected semi-simple  Lie group with finite centre, then $G$ is $z$-minimal. Hence, $\Inn(G)$ is minimal by Proposition \ref{prop:birk}. Corollary \ref{semisimple} shows that the same holds 
even without assuming that the centre is finite.
\end{remark}
The following example shows that not every centre-free connected Lie group is minimal.
\begin{example}\label{ex:solv}
Let $N=\mathbb{C}\times \mathbb{C}$ and $T=\mathbb{T}\times \mathbb{T}$, where $\mathbb{C}$ is the field of complex numbers and $\mathbb{T}$ is the unit circle in the complex plane.
Consider the following continuous action
\begin{equation}\label{e3}T\times N\to N: ((e^{2\pi ir}, e^{2\pi is}), (z_1, z_2))\mapsto (e^{2\pi ir} z_1, e^{2\pi is}z_2).\end{equation}
Fixing an irrational number $h\in \mathbb{R}$, we can define another continuous action: 
\begin{equation}\label{e2}\mathbb{R}\times N\to N: (x, (z_1, z_2))\mapsto(e^{2\pi ix}z_1, e^{2\pi ihx}z_2).\end{equation}
Using the two actions, we get connected Lie groups $G=N\rtimes \mathbb{R}$ and $H=N\rtimes T$.
Now let us see that $G$ is centre-free. If $((z_1,z_2), x)\in Z(G)$, then $(z_1, z_2)$ should be fixed by any $r\in \mathbb{R}$.
However,  this happens only when $z_1=z_2=0$. 
On the other hand, the kernel of the action is trivial since $h$ is irrational. So the only central element of $G$ contained in $\{(0,0)\}\times \mathbb{R}$ is the trivial element.
Thus, $G$ is centre-free.
Moreover, $G$  is not minimal: the mapping
$$\varphi: G\to H, ((z_1, z_2),x)\mapsto ((z_1, z_2), (e^{2\pi i x}, e^{2\pi i hx}))$$
is continuous monomorphism but not open onto its image.
\end{example}
The following corollary provides another instance in which $z$-minimality and minimality coincide. 
\begin{corollary}
 Let $G$ be a connected Lie group with finite centre. Then $G$ is minimal if and only if it is  $z$-minimal. In particular, the arbitrary product of minimal connected Lie groups with finite centre is minimal. 
\end{corollary}
\begin{proof}
If $Z(G)$ is finite, then it is both compact and discrete. By Theorem \ref{Th:Min}, $ \Inn(G)$ is minimal. If $G$ is minimal, then by Proposition \ref{prop:inisq},
$G/Z(G)$ is minimal. By Lemma \ref{lem:zminiff},  $G$ is $z$-minimal.  The last assertion follows from the fact that the product of complete $z$-minimal groups is minimal.
\end{proof}
 We now give a necessary condition for the minimality of $G/Z(G)$, where $G$ is a connected locally compact group.
\begin{proposition}\label{Centfree}
Let $G$ be a connected locally compact group such that $G/Z(G)$ is minimal. Then $G/Z(G)$ is centre-free.
\end{proposition}
\begin{proof}
We first assume that $Z(G)$ is connected and compact.
Let $H=G/Z(G)$ and $q: G\to H$ be the canonical mapping. Then, by the minimality of $H$, we obtain that $Z(H)$ is compact.
By \cite[Theorem 13]{Iwa}, the compact subgroup $Z(H)$ of the connected locally compact group $H$ is contained in a connected compact subgroup $P$, 
which can be chosen to be solvable  by \cite[Theorem 15]{Iwa}.
Let $K=q^{-1}(P)$, then $K$ is a compact connected solvable group.
This implies that $K$ is abelian, since compact connected solvable groups are abelian (see, for example, \cite[Lemma 2.2]{Iwa}).
Note that $Z_2(G)=q^{-1}(Z(H))$ is a normal subgroup of $G$ contained in $K$; so it is abelian.
According to a theorem of Iwasawa \cite[Theorem 4]{Iwa}, compact normal abelian subgroups of a connected topological group must be central.
So, $Z_2(G)\subseteq Z(G)$, that is, $Z(H)=q(Z_2(G))\subseteq q(Z(G))=\{1\}$ and we are done.

Now, let us consider the general case.
Since $Z(G)$ is a  locally compact abelian group, it has a closed totally disconnected subgroup $D$ such that $Z(G)/D$ is connected and compact, see \cite[Proposition 5.43]{HM}.
If $G/Z(G)$ is minimal, then 
 its isomorphic copy $(G/D)/(Z(G)/D)$  is minimal as well, in view of the  \textit{third isomorphism theorem for topological groups}
 (see \cite[Theorem 1.5.18]{AT} and \cite[Proposition 3.6]{It}).
Note that, by Fact \ref{centre} (2), $Z(G/D)=Z(G)/D$.
So by the above case, $(G/D)/Z(G/D)$ is centre-free.
It follows then from
$$G/Z(G)\cong (G/D)/(Z(G)/D)=(G/D)/Z(G/D)$$
that $G/Z(G)$ is also centre-free.
\end{proof}

By Proposition \ref{prop:birk}, if $G/Z(G)$ is minimal, then so is $\Inn(G)$. We will see that, if $G$ is also a Lie group, the necessary condition for the minimality of $G/Z(G)$ becomes a necessary condition for the minimality of  $\Inn(G).$
\begin{theorem}\label{Th:Main2}
Let  $G$ be  a connected Lie group. If $\Inn(G)$ is minimal, then  it is centre-free.
\end{theorem}
\begin{proof}
If $G$ is (CA), then by Proposition \ref{CAiff}, $G/Z(G)\cong\Inn(G)$ is minimal.
Then Proposition \ref{Centfree} can be applied.

Now consider the case when $G$ is not (CA).
 Let $H, \varphi, M, T, V, \sigma'$ be as in Theorem \ref{nonCA}.
Since $\sigma'(H)$ contains $\Inn(G)$ as a dense subgroup and $\Inn(G)$ is minimal, $\sigma'(H)$ must be minimal as well.
By (iv) of Theorem \ref{nonCA}, $\sigma'(H)$ is a closed, hence locally compact subgroup, of $\Aut(G)$.
Then, $\sigma'$ is open onto its image, again by \cite[Theorem 3.1.27]{AT}.
Thus, $\sigma'(H)\cong H/Z(H)$, because $Z(H)$ is the kernel of $\sigma'$: an element $h\in \ker \sigma'$ if and only if it commutes with any $g\in \varphi(G)$, this is equivalent to say that $h$ commutes with every element in $H$ because $\varphi(G)$ is dense in $H$.
So, by Proposition \ref{Centfree}, $\sigma'(H)\cong H/Z(H)$ is centre-free.
Then $\Inn(G)$ is also centre-free.
Indeed, if $x\in \Inn(G)$ commutes with every element in $\Inn(G)$, then it commutes with every element in $\sigma'(H)$ as well, since $\Inn(G)$ is dense in $\sigma'(H)$.
So $x$ can only be the identity.
%
\end{proof}

\begin{corollary}\label{nil}
Let $G$ be a  connected locally compact nilpotent group. Then the following conditions are equivalent:
\begin{itemize}
  \item[(1)] $G$ is abelian;
  \item[(2)] $G/Z(G)$ is minimal.
\end{itemize}
Moreover, if $G$ is also a Lie group, then they are equivalent to
\begin{itemize}
  \item[(3)] $\Inn(G)$ is minimal
\end{itemize}
as well.
\end{corollary}

\begin{proof}
The implication (1) $\Rightarrow$ (2) is trivial.
Let us consider the other direction.
If $G/Z(G)$ is minimal, then by Proposition \ref{Centfree}, $G/Z(G)$ is centre-free.
Since a centre-free nilpotent group can only be the trivial group, $G=Z(G)$ is abelian.

When $G$ is additionally assumed to be a Lie group, we have that $\Inn(G)\cong G/Z(G)$ by Example \ref{ExCA} (3) and Proposition \ref{CAiff}.
So (2) and (3) are equivalent.
\end{proof}

\begin{corollary}\label{cor:ofnilp}
A connected locally compact nilpotent  group is $z$-minimal only if it is compact and abelian.
\end{corollary}
In contrast, as  the following example  shows, there exists a totally minimal (so $z$-minimal, in particular) connected solvable  Lie group that is neither compact nor abelian.  
\begin{example}\cite[Example 3.13]{XS}
By	Mayer \cite[Examples 2.6(i)]{Mayer}, the Euclidean motion group $\R^n \rtimes \SO(n,\R)$ is totally minimal for every $n>1$, where   the special orthogonal group $\SO(n,\R)$ acts on $\R^n$ by matrix multiplication.
	Fixing $n=2$ one obtains  the  Lie group of orientation-preserving isometries of the complex plane \[G:=\ISO_{+}(\C)= \Bigg\{\left(\begin{array}{cc}
		a & b \\
		0 & 1 \\
	\end{array}\right)\bigg | \ \ a \in \T, b\in \C \Bigg\}\cong \C \rtimes \T,\]  where $\T$ acts on $\C$ by multiplication. Note that the total minimality of the non-compact connected metabelian Lie group $G$ can be established in a different way by  Banakh \cite[Theorem 13]{Banakh1}.
\end{example}
\begin{lemma}\label{lem:coniso}
Let $G$ be a locally compact group and $q:G\to G/Z(G)$ be the quotient map. If $G/Z(G)$ is centre-free, then the  map $$\psi: \Inn(G)\to \Inn(G/Z(G)), I_g\mapsto I_{q(g)}$$ is a continuous group isomorphism.
\end{lemma}
\begin{proof}
It is easy to see that if  $G/Z(G)$ is centre-free, then $\psi$ is a group isomorphism. Let us prove the continuity of $\psi$. Consider the action $$\tilde{\alpha}:\Inn(G)\times G/Z(G)\to G/Z(G)$$ defined by      $$\tilde{\alpha}(I_g,q(h))=(q\circ\alpha)(I_g,h) \  \forall (g,h)\in G\times G,$$ where 
$\alpha$ is the natural continuous action of $\Inn(G)$ on $G.$   We show that $\tilde{\alpha}$ is continuos at an arbitrary point $(I_g,q(h)).$ Let $U$ be a neighbourhood of $$\tilde{\alpha}(I_g,q(h))=(q\circ\alpha)(I_g,h).$$ By the continuity of $q$ and $\alpha,$ there exist a neighbourhood $V$ of $I_g$ and a neighbourhood $W$ of $h$ such that
$$\tilde{\alpha}(V\times q(W))=(q\circ\alpha)(V\times W)\subseteq U.$$ This proves the continuity of $\tilde{\alpha},$ as $q(W)$ is a neighbourhood of $q(h).$ Now observe that $\tilde{\alpha}(I_g,q(h))=\beta(I_{q(g)},q(h) )$, where $\beta$ is the natural  action of $\Inn(G/Z(G))$ on $G/Z(G).$ Since the Birkhoff topology is the coarsest group topology on $\Inn(G/Z(G))$ for which $\beta$ is continuous we deduce that $\psi$ is continuous.
\end{proof}
Now we are ready to prove our main result.
\begin{theorem}\label{thm:curmain}
Let    $G$  be a connected Lie group. Then $\Inn(G)$ is minimal if and only if it is centre-free and topologically isomorphic to  $\Inn(G/Z(G)).$ 
\end{theorem}
\begin{proof}
Assume first that  $\Inn(G)$  is centre-free and  $\Inn(G)\cong\Inn(G/Z(G)).$ Then, $G/Z(G)$ is a centre-free connected Lie group. By Proposition \ref{InnMin},
$\Inn(G/Z(G))$ is minimal. It follows that its isomorphic copy $\Inn(G)$ is minimal as well.

Now suppose that $G$ is a connected Lie group such that $\Inn(G)$ is minimal. By  Theorem \ref{Th:Main2},  $\Inn(G)$ is centre-free. By Lemma \ref{lem:coniso}, there exists a continuous group isomorphism $\psi: \Inn(G)\to \Inn(G/Z(G)).$ As $\Inn(G)$ is minimal, $\psi$ is, in fact, a topological group isomorphism.
\end{proof}
Similarly, we have:
\begin{theorem}\label{thm: minofgzg}
Let $G$ be a connected Lie group. Then 	$G/Z(G)$ is minimal if and only if it is centre-free and topologically isomorphic to $\Inn(G/Z(G)).$ 
\end{theorem}
\begin{proof}
Assume first that $G/Z(G)$  is centre-free and  $G/Z(G)\cong\Inn(G/Z(G)).$  By Proposition \ref{InnMin},
$\Inn(G/Z(G))$ and its isomorphic copy $G/Z(G)$ are minimal.

Now suppose that $G$ is a connected Lie group such that $G/Z(G)$ is minimal. By  Proposition \ref{Centfree},  $G/Z(G)$ is centre-free. This and Proposition \ref{prop:birk}  imply that there exists a continuous group isomorphism $\phi: G/Z(G)\to \Inn(G/Z(G)).$ As $ G/Z(G)$ is minimal, $\phi$ is, in fact, a topological group isomorphism.
\end{proof}
As a corollary we obtain a criterion for the $z$-minimality of a connected Lie group.
\begin{corollary}
Let $G$ be a connected Lie group. Then the following conditions are equivalent:
\ben \item $G$ is $z$-minimal;
\item $G$ is $\Inn$-minimal;
\item \ben \item $Z(G)$ is compact and \item $G/Z(G)$ is centre-free   and topologically isomorphic to $\Inn(G/Z(G)).$ \een
\een
\end{corollary}
\begin{proof}
$(1)\Leftrightarrow (2)$ Use Corollary \ref{cor:zifinn}.\\
$(1)\Leftrightarrow (3)$ By Lemma \ref{lem:zminiff}, $G$ is $z$-minimal if and only if  $Z(G)$ is compact and $G/Z(G)$ is minimal. Now use Theorem \ref{thm: minofgzg}.
\end{proof}


\section{Open questions and concluding remarks}

\begin{question}
Is the condition that $\Inn(G)\cong \Inn(G/Z(G))$ in Theorem \ref{thm:curmain} necessary? In other words, can we obtain this isomorphism from the condition that $G/Z(G)$ is centre-free?
\end{question}

It suffices to check if the continuous group isomorphism $\psi: \Inn(G)\to \Inn(G/Z(G)),$ defined in Lemma \ref{lem:coniso}, is open.

 The following question can be viewed as  a particular case of \cite[Question 7.1]{DHPXX}.
\begin{question}
Is an infinite product of locally compact $z$-minimal groups $z$-minimal?
\end{question}

	In view of Proposition \ref{prop:birk}, it is also natural to ask:
\begin{question}
Let $G$ be a locally compact group for which  $\Inn(G)$ is locally compact minimal. Is $\Inn(G)^\kappa$  minimal for every infinite cardinal $\kappa$?
\end{question}
When $G$ is a connected Lie group, the answer is ``yes''. This is because if $\Inn(G)$ is minimal, then it is centre-free, by Theorem \ref{Th:Main2} 
while arbitrary products of  centre-free minimal groups are minimal.

If $G$ is a locally compact (not necessarily Lie) connected group such that $G/Z(G)$ is minimal, then $G/Z(G)$ is centre-free by Proposition \ref{Centfree}  and the topological isomorphism $G/Z(G)\cong\Inn(G/Z(G))$ holds true. So, one may ask if Theorem \ref{thm: minofgzg} holds for a locally compact connected group  even without assuming that it is a Lie group.
\begin{question}
Let $G$ be a connected locally compact group such that $G/Z(G)$  is centre-free and topologically isomorphic to $\Inn(G/Z(G)).$  Must $G/Z(G)$ be minimal?
\end{question}

Recall that a complete minimal abelian group must be compact while there exist locally compact  minimal nilpotent groups that are not compact.
Indeed, using \cite[Theorem 2.11]{MEG95} one can construct many non-compact    minimal  (two-step) nilpotent locally compact groups.
\begin{question}
Let $G$ be a two-step nilpotent locally compact minimal group. Is $G^\kappa$ minimal for every infinite cardinal $\kappa$?
\end{question}


%

 It is still open whether the Fermat and  Mersenne numbers are square-free (see \cite{G}). The following result follows from Corollary \ref{cor:sq}.
 \begin{corollary}\label{coro-square1}
 	Let $\mathbb{F}$ be a subfield of a local field,  $F_n=2^{2^n}+1$  and $M_p=2^p-1$ be a Fermat number and  a Mersenne number,   respectively, where $n\in \N$ and $p$ is a prime number.
 	\ben
 	
 	\item  If $\SL(F_n, \mathbb{F})$ is minimal but not totally minimal, then $F_n$ is not square-free.
 	\item If $\SL(M_p, \mathbb{F})$ is minimal but not totally minimal, then $M_p$ is not square-free.
 	\een
 \end{corollary}
%

	 \vskip 0.3cm
\noindent \textbf{Acknowledgment.}
We thank G. Luk\' acs and M. Megrelishvili for their useful suggestions.
The first listed author acknowledges the support of NSFC grants No. 12301089 and 12271258, and the Natural Science Foundation of the Jiangsu Higher Education Institutions of China (Grant No. 23KJB110017).

\end{document}